\documentclass[12pt,reqno]{amsart}

\usepackage{color}
\usepackage{hyperref}
\usepackage{url}
\usepackage{amsmath}
\usepackage{amssymb}
\usepackage{amstext}
\usepackage{comment}
\usepackage{stmaryrd}
\usepackage{latexsym}
\usepackage{enumitem}
\usepackage[left=3.5cm, right=3.5cm, headheight=14.0pt]{geometry}

\hypersetup{
    colorlinks=true,
    pdfborder={0 0 0},
}

\usepackage{fancyhdr}
\newtheorem{theorem}{Theorem}
\newtheorem{lemma}[theorem]{Lemma}
\newtheorem{proposition}[theorem]{Proposition}
\newtheorem{corollary}[theorem]{Corollary}
\newtheorem{conjecture}[theorem]{Conjecture}

\theoremstyle{definition}
\newtheorem{definition}[theorem]{Definition}
\newtheorem{example}[theorem]{Example}
\newtheorem{remark}[theorem]{Remark}
\newtheorem*{claim*}{Claim}

\renewcommand{\AA}{\mathbb A}

\providecommand{\ord}{\mathop{\rm ord}\nolimits}

\providecommand{\GGb}{\mathbb{G}}
\providecommand{\NNb}{\mathbf{N}}

\renewcommand{\emptyset}{\varnothing}

\newcommand{\fixed}[2][1]{%
  \begingroup
  \spaceskip=#1\fontdimen2\font minus \fontdimen4\font
  \xspaceskip=0pt\relax
  #2%
  \endgroup
}

\begin{document}
\title{Cauchy-Davenport type theorems for semigroups}
\author{Salvatore Tringali}
\address{Texas A\&M University at Qatar, Education City -- PO Box 23874 Doha, Qatar}
\email{salvatore.tringali@qatar.tamu.edu }
%
\thanks{This research was partially supported by the French ANR Project No. ANR-12-BS01-0011 and mostly developed while the author was funded from the European Community's 7th Framework Programme (FP7/2007-2013) under Grant Agreement No. 276487 (project ApProCEM)}
%
%
\begin{abstract}
Let $\mathbb A = (A, +)$ be a (possibly non-commutative) semigroup. For $Z \subseteq A$ we define $Z^\times := Z \cap \mathbb A^\times$, where $\mathbb A^\times$ is the set of the units of $\mathbb A$, and $$\gamma(Z) := \sup_{z_0 \in Z^\times} \inf_{z_0 \ne z \in Z} {\rm ord}(z - z_0).$$ The paper investigates some properties of $\gamma(\cdot)$ and shows the following extension of the Cauchy-Davenport theorem: If $\mathbb A$ is cancellative and $X, Y \subseteq A$, then
$$|X+Y| \ge \min(\gamma(X+Y),|X| + |Y| - 1).$$
This implies a generalization of Kemperman's inequality for torsion-free groups and
strengthens another extension of the Cauchy-Davenport theorem, where $\mathbb{A}$ is a group and $\gamma(X+Y)$ in the above is replaced by the infimum of $|S|$ as $S$ ranges over the non-trivial subgroups of $\mathbb{A}$ (Hamidoune-K\'arolyi theorem).
\end{abstract}

\maketitle

\section{Introduction}
\label{sec:intro}

The present article fits in a broader program of study initiated in \cite{Tring12, Tring11} and aimed to extend some pieces of the theory on Cauchy-Davenport type theorems to the setting of possibly non-commutative or non-cancellative semigroups.

A motivation for this comes from considering that the non-zero elements of a non-trivial unital ring, either commutative or not, are not cancellative, and hence not even closed, under multiplication (unless the ring is a domain).

Another motivation lies in the fact that, even when $\GGb = (G, +)$ is a commutative group, the non-empty subsets of $G$, endowed with the binary operation taking a pair $(X,Y)$ to the sumset $\{x+y: x \in X, y \in Y\}$, are not a cancellative monoid (unless $\GGb$ is trivial): When $\mathbb G$ is, up to an isomorphism, the additive group of the integers, the corresponding structure on the powerset of $G$ has been studied by J.~Cilleruello, Y.~O.~Hamidoune and O.~Serra \cite{Cill10} in the context of acyclic monoids, and more recently by A.~Geroldinger and W.~Schmid \cite{GerSch} in relation to the factorization theory of (commutative and cancellative) unital semigroups.

Here, more specifically, we shall prove an extension (Theorem \ref{th:new_strong_main}) of the (classical) Cauchy-Davenport theorem (Theorem \ref{th:basic_CD}) to the setting of cancellative, but not necessarily commutative, semigroups. From this we derive as an almost immediate consequence a strengthening (Corollary \ref{cor:karolyi_generalized}) of an addition theorem for groups due to H.~O.~Hamidoune and G.~K\'arolyi (Theorem \ref{th:karolyi_finite}).

In addition, we present and discuss some aspects of a conjecture (Conjecture \ref{conj:cauchy_davenport}) which, if true, would further improve most of the results in the paper.

In all of this, a key role is played by certain properties of what we call the Cauchy-Davenport constant of a tuple of sets (Definition \ref{def:davenport_constant}), which are also investigated to some extent in this work.
\section{Preliminaries and general notation}
\label{sec:preliminaries}
We refer to \cite{Ruzsa09} for standard notation and terminology from additive [semi]group theory (note that, in this paper, ``additive'' does not imply ``commutative'').

We denote by $\NNb$ the set of non-negative integers, endowed with its usual operations and order, which we extend as customary to $\NNb \cup \{\infty\}$ by adjoining an element $\infty \notin \NNb$ and taking, in particular, $0 \cdot \infty :=\infty \cdot 0 := 0$.

At several points, we will use without explicit mention that if $A \subseteq B \subseteq \NNb \cup \{\infty\}$ then $\inf(B) \le \inf(A)$ and $\sup(A) \le \sup(B)$, with the convention that the supremum of $\emptyset$ is $0$ and the infimum of $\emptyset$ is $\infty$.

Unless differently stated, $i$ and $n$ shall denote positive integers. Given a set $X$, we write $|X|$ for the size of $X$ if $X$ is finite, and let $|X| := \infty$ otherwise.

Here, a semigroup is a pair $\mathbb A = (A, +)$ consisting of a set $A$ and an associative binary operation $+$ on $A$ (we refer to \cite{Howie96} for basic aspects of semigroup theory).

We let $\mathbb A^\times$ be the set of units of $\mathbb A$, so $\mathbb A^\times = \emptyset$ if $\mathbb A$ is not a \textit{monoid} (or unital semigroup).
In this respect, we recall that, if $\mathbb A$ is unital with identity $0$, a \textit{unit} of $\mathbb A$
is an element $z \in A$ for which there exists a (provably unique) $\tilde z \in A$, called
the inverse of $z$ (in $\mathbb A$), such that $z + \tilde z =
\tilde z + z = 0$ (we usually denote $\tilde{z}$ by $-z$).

For $Z \subseteq A$ we write $Z^\times$ instead of $Z \cap \mathbb A^\times$ (if there is no likelihood of confusion) and $\langle Z \rangle_\mathbb{A}$ for the smallest sub\textit{semigroup} of $\mathbb A$
containing $Z$.
Then given $z \in A$, we use
$\ord_\mathbb{A}(z)$ for the \textit{order} of $z$ (in $\mathbb A$), i.e. we let $\ord_\mathbb{A}(z)
:= |\langle \{z\} \rangle_\mathbb{A}|$, so generalizing the common notion of order for the elements of a group.

An element $z \in A$ is said to be cancellable (in $\mathbb A$) if $x+z = y+z$ or $z + x = z + y$ for $x,y \in A$ imply $x = y$, and $\mathbb A$ is cancellative if each $z \in A$ is cancellable.

Building on these premises, we assume for the remainder that $\mathbb A = (A, +)$ is a fixed, arbitrary semigroup (unless differently noted), and let $0$ be the identity of a \textit{unitization}, $\mathbb A^{(0)}$, of $\mathbb A$:
If $\mathbb A$ is unital, then $\mathbb A^{(0)} := \mathbb A$; otherwise, $\mathbb A^{(0)} := (A \cup \{0\}, +)$, where $0$ is some element $\notin A$ and, by an abuse of notation, $+$ is the unique extension of the operation of $\mathbb A$ to a binary operation on $A \cup \{0\}$ for which $0$ serves as an identity.
Then, for $S \subseteq A$ we define $\mathfrak{p}_\mathbb{A}(S) := \inf_{z \in S \setminus \{0\}} \ord_{\mathbb A^{(0)}}(z)$.

If no confusion can arise, we omit the subscript `$\mathbb A$' from this notation, and for $X \subseteq A$ and $z \in A$ we write $X+z$ in place of $X+\{z\}$, and use $X-{z}$ for $X+\tilde{z}$ if $z \in A^\times$ and $\tilde{z}$ is the inverse of $z$ (similarly with $z+X$ and ${-z}+X$).

\section{Cauchy-Davenport type theorems}
\label{sec:the_CD_theorem}

The Cauchy-Davenport theorem, see
\cite{Cauchy1813} and
\cite{Daven47}, is one of the first significant achievements in the field of additive theory:
\begin{theorem}[Cauchy-Davenport theorem]
\label{th:basic_CD}
Let $\mathbb A$ be a group of prime order $p$
and $X,Y$ non-empty subsets of $A$. Then $|X+Y| \ge \min(p,
|X|+|Y|-1)$.
\end{theorem}
The theorem applies especially to the
additive group of the integers modulo a prime.
Many extensions to composite moduli are known, and a couple of them, due to I. Chowla and S.~S. Pillai, have been recently sharpened in \cite{Tring12}
as a byproduct of Theorem \ref{th:old_strong_main} below.
In the same paper, there also appears the following definition (though in a different notation), which lies at the heart of the present work:
\begin{definition}\label{def:davenport_constant}
For a set $X \subseteq A$, we let
$$\gamma_\mathbb{A}(X) := \sup_{x_0 \in X^\times}
\inf_{x_0 \ne x \in X} \ord(x - x_0)
$$
Then, for $X_1,\ldots, X_n \subseteq A$ we define
$$
\gamma_\mathbb{A}(X_1,\ldots, X_n) := \max(\gamma_\mathbb{A}(X_1), \ldots, \gamma_\mathbb{A}(X_n)),,
$$
and refer to $\gamma_\mathbb{A}(X_1,\ldots, X_n)$ as the Cauchy-Davenport constant of the $n$-tuple $(X_1, \ldots, X_n)$ relative to $\mathbb A$ (we omit the subscript `$\mathbb A$' if clear from the context).
\end{definition}
Any tuple of subsets of $A$ has a well-defined
Cauchy-Davenport constant (relative to $\mathbb A$), and it is interesting to compare it with other ``structural parameters'':
\begin{lemma}\label{lem:simple_lemma}
Let $X,Y$ be subsets of $A$ and assume that $\mathbb A$ is cancellative and $X^\times + Y^\times$ is non-empty. Then $
\gamma(X,Y) \ge \min(\gamma(X), \gamma(Y)) \ge \gamma(X+Y) \ge \mathfrak p(A)$.
\end{lemma}
Lemma \ref{lem:simple_lemma} is proved by the end of Section \ref{sec:lemmas} and
applies, on the level of groups, to \textit{any} pair of non-empty subsets. The following example suggests that the  inequalities in the lemma are, in general, rather pessimistic:
\begin{example}
\label{exa:weaker}
Fix an integer $m \ge 2$ and prime numbers $p$, $q$ with $ m < p < q$. Then set $n := m pq$, $X := \{mk \bmod n: k = 0, \ldots, p-1\}$ and $Y := \{mk \bmod n: k = 1, \ldots, p\}$. We have $|X+Y| = 2p$, $\gamma(X) = \gamma(Y) = p q$ and $\gamma(X+Y) = q$, while $\mathfrak p(\mathbb Z/n\mathbb Z)$ is the smallest prime divisor of $m$, where $\mathbb{Z}/n\mathbb{Z}$ is the additive group of the integers modulo $n$. Thus, $\gamma(X,Y) > \gamma(X+Y) > \mathfrak p(\mathbb Z/n\mathbb Z)$, and indeed $ \gamma(X+Y) \gg \mathfrak{p}(\mathbb Z/n\mathbb Z)$ for $q \gg m$, and $\gamma(X,Y) \gg \gamma(X+Y)$ for $p \gg 2$, where the symbol $\gg$ reads ``(comparatively) much larger than''.
\end{example}
Note that, for $X \subseteq A$, it holds $\gamma(X) = 0$ if $X^\times$ is empty, and then $\gamma(\cdot)$ provides no significant information on $\mathbb A$. However, this is not the case, e.g., when $X \ne \emptyset$ and $\mathbb A$ is a group (with the result that $X^\times = X$), which is the ``moral basis'' for the non-trivial bounds below. To start with, we have:
\begin{theorem}\label{th:old_strong_main}
Suppose $\mathbb A$ is cancellative and let $X,Y$ be non-empty finite subsets of $A$ such that $\langle Y \rangle$ is commutative. Then $|X+Y| \ge
\min(\gamma(Y),|X|+|Y|-1)$.
\end{theorem}
The result appears as Theorem 8 in \cite{Tring12} and leads us to the following:
\begin{conjecture}
\label{conj:cauchy_davenport}
Let $X_1,\ldots, X_n$ be subsets of $A$ with $|X_i| \ge 2$ for each $i$. If $\mathbb A$ is cancellative, then $|X_1+\cdots+X_n| \ge \min(\gamma(X_1,\ldots, X_n),|X_1|+\cdots+|X_n|+1-n)$.
\end{conjecture}
We do not know how to prove the conjecture, which can however be confirmed in some special cases, see Corollary \ref{cor:kemperman_generalized} below or consider Theorem \ref{th:old_strong_main}  when $\mathbb A$ is commutative. (Note that the restriction on the size of the summands is necessary, as otherwise the conjecture is easily disproved for $n \ge 3$.) Instead, we have the following theorem, which is the main contribution of the paper:
\begin{theorem}\label{th:new_strong_main}
Let $X, Y$ be subsets of $ A$ and suppose that $\mathbb A$ is cancellative. Then $|X+Y| \ge \min(\gamma(X+Y), |X| + |Y| - 1)$.
\end{theorem}
It is worth comparing Theorems \ref{th:old_strong_main} and \ref{th:new_strong_main}: On the one hand, the latter is ``much stronger'' than the former, for it does no longer depend on commutativity. Yet on the other hand, the former is ``much stronger'' than the latter, since for $X,Y \subseteq A$ we are now replacing $\gamma(X,Y)$ in Theorem \ref{th:old_strong_main} with $\gamma(X+Y)$, and it has been already observed (Example \ref{exa:weaker}) that this means, in general, a notably weaker bound.

With that said, Theorem \ref{th:new_strong_main} is already strong enough to allow for a strengthening of the following result, as implied by Lemma \ref{lem:simple_lemma} and Example \ref{exa:weaker}:
\begin{theorem}[Hamidoune-K\'arolyi theorem]
\label{th:karolyi_finite}
If $\mathbb A$ is a group and $X,Y$ are non-empty subsets of
$A$, then $|X+Y| \ge \min(\mathfrak{p}(A), |X| + |Y|-1)$.
\end{theorem}
The result is more or less straightforward in the commutative case (by Kneser's theorem), and G. K\'arolyi proved it for finite groups in 2005 (based on the Feit-Thompson theorem). A proof of the general statement (relying on the isoperimetric method) was then communicated by H.~O.~Hamidoune to K\'arolyi during the peer-review process of \cite{Karo05.2}, where it was finally included, see \cite[p. 242]{Karo05.2}.
However, K\'arolyi himself pointed out to the author, as recently as July 2013, that a simpler approach comes from a Kneser-type result of J.~E.~Olson \cite[Theorem 2]{Olson}, based on Kemperman's transform. And another argument along the same lines
was mentioned by I.~Ruzsa in a private communication in June 2013.

As a matter of fact, also our proof of Theorem \ref{th:karolyi_finite} is basically a transformation proof, close in the spirit to Olson's proof, and it comes as a  consequence of Theorem \ref{th:new_strong_main} in view of Lemma \ref{lem:simple_lemma}. Specifically, we have:
\begin{corollary}
\label{cor:karolyi_generalized}
Pick subsets $X_1,\ldots, X_n$ of $A$ such that $X_1^\times+\cdots+X_n^\times \ne \emptyset$. If $\mathbb A$ is cancellative, then $|X_1 + \cdots + X_n| \ge \min(\mathfrak p(A), |X_1| + \cdots + |X_n| + 1 - n)$.
\end{corollary}
Theorem \ref{th:new_strong_main} and Corollary \ref{cor:karolyi_generalized} are proved in Section \ref{sec:main_results}. As a consequence, we get:
\begin{corollary}
\label{cor:kemperman_generalized}
Given $X_1,\ldots, X_n \subseteq A$ such that $X_1^\times+\cdots+X_n^\times \ne \emptyset$, let $\kappa := |X_1| + \cdots + |X_n| + 1 - n$ and assume $\ord(x) \ge \kappa$ for every $x \in A \setminus \{0\}$. If $\mathbb A$ is cancellative, then $|X_1 + \cdots + X_n| \ge \kappa$.
\end{corollary}
This is immediate by Corollary \ref{cor:karolyi_generalized} and generalizes an inequality by J.~H.~B.~Kemperman, see \cite[p. 251]{Kemp56}, a special case of which, see \cite[Corollary 3.3.3]{Hami09}, is often referred to as Kemperman's inequality for torsion-free groups. (Incidentally, \cite{Kemp56} is mainly focused on cancellative
semigroups, there simply called semigroups.)
\section{Properties of the Cauchy-Davenport constant}
\label{sec:lemmas}
Throughout, we derive properties of the Cauchy-Davenport constant, most of which will be used later to
prove Theorem \ref{th:new_strong_main} and Corollary \ref{cor:karolyi_generalized}. In this sense, some results of the present section (in particular, Corollary \ref{cor:units_give_invariant_n_transform}) could have been stated and proved in a less general form. However, we believe that would not have really simplified the exposition, and at the same time we hope that, in the way they are actually presented, they can help to prove or refine Conjecture \ref{conj:cauchy_davenport}.

The following proposition, on the other hand, is readily adapted from the case of groups, but we have no standard reference for it in the context of semigroups.
\begin{proposition}\label{prop:trivialities}
Let $X_1, Y_1,\ldots, X_n, Y_n, Z$ be subsets of $A$. Then:
\begin{enumerate}[label={\rm (\roman{*})}]
\item\label{item:trivialieties_(i)} If $X_i \subseteq Y_i$ for each $i$, then $\sum_{i=1}^n X_i \subseteq \sum_{i=1}^n Y_i$ and $\big|\sum_{i=1}^n X_i\big| \leq \big|\sum_{i=1}^n Y_i\big|$.
\item\label{item:trivialieties_(ii)} $\big|\sum_{i=1}^n X_i\big| \le \prod_{i=1}^n |X_i|$, and in addition $\big|\sum_{i=1}^n X_i\big| \ge \max(|X_1|, \ldots, |X_n|)$ if $X_i$ contains at least one cancellable element for each $i$.
\item\label{item:trivialieties_(iib)} If $z \in A$ is cancellable, then $|x+Z|=|Z+x|=|Z|$.
\item\label{item:trivialieties_(iii)} If $z_1, \ldots, z_n \in A^\times$, then $\big|\sum_{i=1}^n X_i\big| = \big|\sum_{i=1}^n (z_{i-1} + X_i - z_i)\big|$.
\end{enumerate}
\end{proposition}
\begin{proof}
\ref{item:trivialieties_(i)} is trivial, while all the rest follows from considering that units are cancellable elements and, for a cancellable element $z \in A$, both of the functions $A \to A: x \mapsto x + z$ and $A \to A: x \mapsto z + x$ are bijective.
\end{proof}
The next lemma and the subsequent remark will prove useful at various points (in particular, the general case with $n$ summands is used in Corollary \ref{cor:karolyi_generalized}).
\begin{lemma}\label{prop:units_distribution}
If $X_1, \ldots, X_n \subseteq A$, then $X_1^\times + \cdots + X_n^\times \subseteq (X_1 + \cdots + X_n)^\times$, and the inclusion is, in fact, an equality if $\mathbb A$ is cancellative.
\end{lemma}
\begin{proof}
The assertion is obvious for $n = 1$, so it is enough to prove it for $n = 2$, since then the conclusion follows by induction. For, let $X,Y$ be subsets of $A$.

Suppose first that $z \in X^\times + Y^\times$ (which means, in particular, that $\mathbb A$ is unital), viz. there exist $x \in X^\times$ and $y \in Y^\times$ such that $z = x + y$. If $\tilde x$ is the inverse of $x$ (in $\mathbb A$) and $\tilde y$ is the inverse of $y$, then it is immediate that $\tilde y + \tilde x$ is the inverse of $x + y$, and hence $x+y \in (X+Y)^\times$. It follows that $X^\times + Y^\times \subseteq (X + Y)^\times$.

Now assume that $\mathbb A$ is cancellative and pick $z \in (X + Y)^\times$; we have to show that $z \in X^\times + Y^\times$. Let $\tilde z$ be the inverse of $z$, and pick $x \in X$ and $y \in Y$ such that $z = x+y$. We define $\tilde x := y + \tilde z$ and $\tilde y := \tilde z + x$. It is seen that $x + \tilde x = (x+y) + \tilde z = 0$ and $\tilde y + y = \tilde z + (x+y) = 0$. Also, $(\tilde x + x) + y = y + \tilde z + (x + y) = y$ and $x + (y + \tilde y) = (x + y) + \tilde z + x = x$, from which we get, by cancellativity, $\tilde x + x = y + \tilde y = 0$. This implies that $z \in X^\times + Y^\times$, and so we are done.
\end{proof}
\begin{remark}
\label{rem:inverse_of_sum_of_units}
The proof of Lemma \ref{prop:units_distribution} shows that, if $x_1, \ldots, x_n \in A^\times$ and $\tilde x_i$ is the inverse of $x_i$, then $\tilde x_n + \cdots + \tilde x_1$ is the inverse of $x_1 + \cdots + x_n$. This is a standard fact about groups, which goes through verbatim for monoids.
\end{remark}
The next proposition proves that the Cauchy-Davenport constant of a set is invariant under translation by units.
\begin{proposition}\label{prop:unital_shifts_vs_gamma_equality}
Let $z \in A^\times$ and $X \subseteq A$. Then $\gamma(X) = \gamma(X + z) = \gamma(z + X)$.
\end{proposition}
\begin{proof}
Denote by $\tilde z$ the inverse of $z$. It is enough to show that $\gamma(X) \le \gamma(X + z)$, as this will in turn imply that $
\gamma(X) \le \gamma(X + z) \le \gamma((X + z) + \tilde z) = \gamma(X)$. (The other equality is proved similarly, and we omit further details.)

For, it follows from Lemma \ref{prop:units_distribution} that $X^\times + z \subseteq (X + z)^\times$, and thus
\begin{equation}
\label{equ:unital_2}
\gamma(X + z) = \sup_{w_0 \in (X + z)^\times} \inf_{w_0 \ne w \in X + z} \ord(w - w_0)
              \ge \sup_{w_0 \in X^\times + z} \inf_{w_0 \ne w \in X + z} \ord(w - w_0).
\end{equation}
But $w \in X + z$ if and only if $w = x + z$ for some $x \in X$, and in fact $w \in X^\times + z$ if and only if $x \in X^\times$. Also, given $x_0 \in X^\times$ and $x \in X$, it holds $x + z = x_0 + z$ if and only if $x = x_0$. It is hence immediate from \eqref{equ:unital_2} and Remark \ref{rem:inverse_of_sum_of_units} that
\begin{equation*}
\gamma(X + z) \ge \sup_{x_0 \in X^\times} \inf_{x_0+z \ne w \in X+z} \ord(w + \tilde z - x_0)
              = \sup_{x_0 \in X^\times} \inf_{x_0 \ne x \in X} \ord(x - x_0) = \gamma(X),
\end{equation*}
which is enough to complete the proof.
\end{proof}
Now we define an \textit{invariant $n$-transform} of $\mathbb A$ to be any $n$-tuple $(T_1,\ldots, T_n)$ of functions on the powerset of $A$, herein denoted by $\mathcal P(A)$, such that
$$|X_1+\cdots+X_n| = |T_1(X_1) + \cdots + T_n(X_n)|$$
and
$${\gamma(X_1+\cdots+X_n)} = {\gamma(T_1(X_1) + \cdots + T_n(X_n))}$$
for all non-empty $X_1, \ldots, X_n \in \mathcal P(A)$.

An interesting case occurs when each $T_i$ is a function of the type
$\mathcal P(A) \to \mathcal P(A): X \to z_l + X + z_r$
for some $z_l, z_r \in A^\times$, as is implied by the following corollary, which ultimately reduces Conjecture \ref{conj:cauchy_davenport} and Theorem \ref{th:new_strong_main} to a ``normalized form''.
\begin{corollary}
\label{cor:units_give_invariant_n_transform}
Let $\mathbb A$ be a monoid, and let $X_1, \ldots, X_n \subseteq A$ such that $X_1^\times + \cdots + X_n^\times \ne \emptyset$. There then exists an invariant $n$-transform $\mathbf T = (T_1, \ldots, T_n)$ such that $0 \in \bigcap_{i=1}^n T_i(X_i)$. Moreover, if $\mathbb A$ is cancellative and $X_1^\times + \cdots + X_n^\times$ is finite, then $\bf T$ can be chosen in such a way that
\begin{equation}
\label{equ:claimed_invariance}
\gamma(T_1(X_1) + \cdots + T_n(X_n)) = \min_{0 \fixed[0.15]{ \text{ }} \ne \fixed[0.15]{ \text{ }} w \fixed[0.15]{ \text{ }} \in \fixed[0.15]{ \text{ }} T_1(X_1) + \cdots + T_n(X_n)} \ord(w).
\end{equation}
\end{corollary}
\begin{proof}
For each $i = 1, \ldots, n$ pick a unit $x_i \in X_i^\times$, using that $X_1^\times + \cdots + X_n^\times$ is non-empty (and hence $X_i^\times \ne \emptyset$), and let $T_i$ be the function $\mathcal P(A) \to \mathcal P(A): X \mapsto z_{i-1} + X - z_i$,
where $z_0 := 0$ and $z_i := x_1 + \cdots + x_i = z_{i-1} + x_i$.

Then clearly $0 \in \bigcap_{i=1}^n T_i(X_i)$, and by construction $\sum_{i=1}^n T_i(X_i) = z_0 + (X_1 + \cdots + X_n) + z_n$. Thus, we get by Proposition \ref{prop:trivialities}.\ref{item:trivialieties_(iii)} that
\begin{displaymath}
\textstyle |X_1| = |T_1(X_1)|, \ldots, |X_n| = |T_n(X_n)|\ \ \text{and} \ \ \big|\sum_{i=1}^n X_i\big| = \big|\sum_{i=1}^n T_i(X_i)\big|,
\end{displaymath}
while Proposition \ref{prop:unital_shifts_vs_gamma_equality} implies $\gamma(X_i) = \gamma(T_i(X_i))$ for each $i$ and $
\gamma(X_1 + \cdots + X_n) = \gamma(T_1(X_1) + \cdots + T_n(X_n))$. This proves the first part of  the claim.

As for the rest, assume in what follows that $\mathbb A$ is cancellative and $X_1^\times + \cdots + X_n^\times$ is finite. Then, letting $Z := X_1 + \cdots + X_n$ for brevity yields, by Proposition \ref{prop:units_distribution}, that $X_1^\times + \cdots + X_n^\times = Z^\times$, so there exist $\bar x_1 \in X_1, \ldots, \bar x_n \in X_n$ such that
\begin{equation}
\label{equ:invariance_done}
\gamma(Z) = \min_{\bar z \fixed[0.15]{ \text{ }} \ne \fixed[0.15]{ \text{ }} z \fixed[0.15]{ \text{ }} \in \fixed[0.15]{ \text{ }} Z} \ord(z - \bar z),
\end{equation}
where $\bar z := \bar x_1 + \cdots + \bar x_n$ and we are using that a supremum taken over a non-empty finite set is, in fact, a maximum.

It follows  that we can define an invariant $n$-transform $\bar{\bf T} = (\bar T_1, \ldots, \bar T_n)$ such that $0 \in \bigcap_{i=1}^n \bar T_i(X_i)$ and $\sum_{i=1}^n \bar{T}_i(X_i) = Z - \bar z$, so that
\begin{displaymath}
\gamma(Z) = \gamma(Z - \bar z) \ge \min_{0 \fixed[0.15]{ \text{ }} \ne \fixed[0.15]{ \text{ }} w \fixed[0.15]{ \text{ }} \in \fixed[0.15]{ \text{ }} Z - \bar z} \ord(w) = \min_{\bar{z} \fixed[0.15]{ \text{ }} \ne \fixed[0.15]{ \text{ }} z \fixed[0.15]{ \text{ }} \in \fixed[0.15]{ \text{ }} Z} \ord(z - \bar z),
\end{displaymath}
by the invariance of $\bar{\bf T}$ and the fact that, on the one hand, $0 \in Z - \bar z$ and, on the other hand, $w \in Z - \bar z$ if and only if $w = z - \bar z$ for some $z \in Z$. Together with \eqref{equ:invariance_done}, this ultimately leads to $ \gamma(Z - \bar z) = \min_{0 \fixed[0.15]{ \text{ }} \ne \fixed[0.15]{ \text{ }} w \fixed[0.15]{ \text{ }} \in \fixed[0.15]{ \text{ }} Z - \bar z} \ord(w)$, and thus to \eqref{equ:claimed_invariance}.
\end{proof}
We conclude the section with a proof of Lemma \ref{lem:simple_lemma}:
\begin{proof}[Proof of Lemma \ref{lem:simple_lemma}]
It is enough to prove that $\gamma(Y) \ge \gamma(X+Y) \ge \mathfrak p(A)$, since all the rest is more or less trivial from our definitions or by repeating the same reasoning with $\gamma(X)$ in place of $\gamma(Y)$.

For, pick $z_0 \in (X + Y)^\times$ using that, on the one hand, $(X+Y)^\times = X^\times + Y^\times$ by Proposition \ref{prop:units_distribution} and the cancellativity of $\mathbb A$, and on the other hand, $X^\times + Y^\times$ is non-empty by the standing assumptions. There then exist $x_0 \in X^\times$ and $y_0 \in Y^\times$ such that $z_0 = x_0 + y_0$, and it is immediate from Remark \ref{rem:inverse_of_sum_of_units} that, for all $y \in A$,
\begin{displaymath}
\langle x_0 + y - z_0 \rangle = x_0 + \langle y - y_0 \rangle - x_0,
\end{displaymath}
which, together with Proposition \ref{prop:trivialities}.\ref{item:trivialieties_(iii)}, gives $
\ord(y - y_0) = \ord(x_0 + y - z_0)$.
So considering that, for $y \in A$, it holds $x_0 + y = z_0$ if and only if $y = y_0$, we get
\begin{displaymath}
\label{lem:simple_lemma_1}
\inf_{y_0 \ne y \in Y} \ord(y - y_0) = \inf_{y_0 \ne y \in Y} \ord(x_0 + y - z_0) \ge \inf_{z_0 \ne z \in X + Y} \ord(z - z_0) \ge \mathfrak p(A),
\end{displaymath}
and this implies the claim by taking the supremum over the units of $X+Y$.
\end{proof}
\section{The proof of the main theorem}
\label{sec:main_results}
At long last, we are ready to prove the central contributions of the paper.
\begin{proof}[Proof of Theorem \ref{th:new_strong_main}]
The claim is obvious if $(X+Y)^\times = \emptyset$, since then
$\gamma(X + Y ) = 0$. So suppose for the remainder of the proof that $(X+Y)^\times$ is non-empty (which implies that $\mathbb A$ is a monoid), and set $\kappa := |X+Y|$, while noticing that, by Lemma \ref{prop:units_distribution}, both $X^\times$ and $Y^\times$ are non-empty, and so, by Proposition \ref{prop:trivialities}.\ref{item:trivialieties_(ii)}, we have
\begin{equation}
\label{equ:kappa}
\kappa \ge \max(|X|,|Y|) \ge \min(|X|,|Y|) \ge 1.
\end{equation}
The statement is still trivial if $\kappa = \infty$ (respectively, $\kappa = 1$), since then either of $X$ or $Y$ is infinite (respectively, both of $X$ and $Y$ are singletons), and hence $|X+Y| = |X| + |Y| - 1$ by \eqref{equ:kappa}.
In what follows, we thus let $\kappa$ be a positive integer and argue by strong induction on $\kappa$, by supposing to a contradiction that $\kappa < \min(\gamma(X+Y), |X|+|Y|-1)$. Based on the above, this means that
\begin{equation}
\label{equ:absurd_hp}
2 \le \kappa < \infty,\ \ 2 \le |X|,|Y| < \infty,\ \ \kappa < \gamma(X+Y),\ \ \text{and}\ \ \kappa \le |X| + |Y| - 2.
\end{equation}
More specifically, there is no loss of generality in assuming, as we do, that $(X,Y)$ is a ``minimax counterexample'' to the claim, by which we mean that, if $(\bar X, \bar Y)$ is another pair of subsets of $A$ with $\bar X^\times + \bar Y^\times \ne \emptyset$ and $|\bar X + \bar Y| < \min(\gamma(\bar X + \bar Y), |\bar X| + |\bar Y| - 1)$, then either $\kappa = |\bar X + \bar Y|$ and one of the following holds:
\begin{equation}
\label{equ:alternative_for_minimax}
\text{(i) } |\bar X| + |\bar Y| < |X| + |Y|; \qquad \text{(ii) } |\bar X| + |\bar Y| = |X| + |Y| \text{ and } |\bar X| \le |X|,
\end{equation}
or $\kappa < |\bar X + \bar Y|$. This makes sense since if $\bar X, \bar Y \subseteq A$, $|\bar X + \bar Y| = \kappa$,  and $\bar X^\times + \bar Y^\times \ne \emptyset$ then $\bar X^\times$ and $\bar Y^\times$ are non-empty, so we get, as before with \eqref{equ:kappa}, that
\begin{displaymath}
|\bar X| \le |\bar X| + |\bar Y| \le 2 \cdot \max(|\bar X|, |\bar Y|) \le 2 \cdot |\bar X + \bar Y| = 2\kappa < \infty.
\end{displaymath}
Finally, in the light of Corollary \ref{cor:units_give_invariant_n_transform}, we may also assume without restriction of generality, up to an invariant $2$-transform, that
\begin{equation}
\label{equ:special_form_of_gamma}
0 \in X \cap Y \ \ \text{and}\ \ \gamma(X + Y) = \min_{0 \ne z \in X + Y} \ord(z).
\end{equation}
Then both $X$ and $Y$ are subsets of $X + Y$, and by the inclusion-exclusion principle we have $\kappa \ge |X| + |Y| - |X \cap Y|$, which gives, together with \eqref{equ:absurd_hp}, that $X \cap Y$ has at least one element different from $0$, i.e. $
|X \cap Y| \ge 2$.

Upon these premises we prove the following intermediate claim (from here on, we set $Z := X \cap Y$ for notational convenience):
\begin{claim*}
There exists $n$ such that $X + nZ + Y \not\subseteq X + Y$, but $X + kZ + Y \subseteq X + Y$ for each $k = 0, \ldots, n-1$, with the convention that $0Z := \{0\}$.
\end{claim*}
\begin{proof}[Proof of the claim]
Assume to the contrary that $X + nZ + Y \subseteq X + Y$ for all $n$. Then we get from $\langle Z \rangle = \bigcup_{n=1}^\infty nZ$ that $X + \langle Z \rangle + Y \subseteq X + Y$, which implies by \eqref{equ:special_form_of_gamma} that $\langle Z \rangle = 0 + \langle Z \rangle + 0 \subseteq X + Y$. So using that $|Z| \ge 2$ to guarantee that $\{0\} \subsetneq Z \subseteq X+Y$, we get by Proposition \ref{prop:trivialities}.\ref{item:trivialieties_(i)} and the same Equation \eqref{equ:special_form_of_gamma} that
\begin{displaymath}
\kappa \ge |\langle Z \rangle| \ge \max_{0 \ne z \in Z} \ord(z) \ge \min_{0 \ne z \in Z} \ord(z) \ge \min_{0 \ne z \in X + Y} \ord(z) = \gamma(X+Y).
\end{displaymath}
This is, however, absurd, for it is in contradiction to \eqref{equ:absurd_hp}, and we are done.
\end{proof}
So, let $n$ be as in the above claim and fix an element $\bar z \in nZ$ such that $X + \bar z + Y \not\subseteq X + Y$ (this exists by construction since otherwise we would have $X + nZ + Y \subseteq X + Y$, which is false). Consequently, observe that
\begin{equation}
\label{equ:z_barred_is_in_the_intersection}
(X + \bar z) \cup (\bar z + Y) \subseteq X + Y.
\end{equation}
In fact, $\bar z$ being an element of $nZ$ entails that there exist $z_1, \ldots, z_n \in Z$ such that $\bar z = z_1 + \cdots + z_n$, whence we get that both of $X + \bar z$ and $\bar z + Y$ are contained in $X + (n-1)Z + Y$. But $X + (n-1)Z + Y $ is, again by construction, a subset of $X+Y$, so \eqref{equ:z_barred_is_in_the_intersection} is proved. With this in hand, let us introduce the sets
\begin{equation*}
X_0 := \{x \in X: x + \bar z + Y \not\subseteq X + Y\}
\end{equation*}
and
\begin{equation*}
Y_0 := \{y \in Y: X + \bar z + y \not\subseteq X + Y\}.
\end{equation*}
It is clear that $X$ (respectively, $Y$) is disjoint from $X_0 + \bar z$ (respectively, from $\bar z + Y_0$). In addition, since $X + \bar z + Y \not\subseteq X + Y$, it is also immediate that $X_0$ and $Y_0$ are both non-empty. Finally, it follows from \eqref{equ:z_barred_is_in_the_intersection} that $0$ is not an element of either $X_0$ or $Y_0$. To sum it up,
\begin{equation}
\label{equ:resume_su_X0_and_Y0}
X_0 \ne \emptyset \ne Y_0,\ \ 0 \notin X_0 \cup Y_0,\ \ \text{and}\ \ (X_0 + \bar z) \cap X = (\bar z + Y_0) \cap Y = \emptyset.
\end{equation}
Now,  we get by Proposition \ref{prop:trivialities}.\ref{item:trivialieties_(iib)} and the cancellativity of $\mathbb A$ that
\begin{equation}
\label{equ:last}
|X_0 + \bar z| = |X_0| =: n_X\ \ \text{and}\ \ |\bar z + Y_0| = |Y_0| =: n_Y,
\end{equation}
which leads to distinguish between the following two cases:
\vskip 0.1cm
\begin{description}[leftmargin=0.6cm]
\item[Case 1] $n_X \ge n_Y$. We form $\bar X$ as the union of $X$ and $X_0 + \bar z$ and $\bar Y$ as the relative complement of $Y_0$ in $Y$. First, note that $0 \in \bar X^\times \cap \bar Y^\times$ by \eqref{equ:resume_su_X0_and_Y0}. Secondly, pick $\bar x \in \bar X$, $\bar y \in \bar Y$ and set $z := \bar x + \bar y$. If $\bar x \in X$, then obviously $z \in X + Y$;
    otherwise, by the construction of $\bar X$ and $\bar Y$, $\bar x \in X_0 + \bar z \subseteq X + \bar z$ and $\bar y \notin Y_0$, so that $\bar x + \bar y \in X + Y$. Therefore, $\emptyset \ne \bar X + \bar Y \subseteq X + Y$ and $0 \in \bar X + \bar Y$, so on the one hand $|\bar X + \bar Y| \le \kappa$ and on the other hand we have by \eqref{equ:special_form_of_gamma} that
    \begin{displaymath}
    \gamma(X+Y) \le \inf_{0 \ne z \in \bar X + \bar Y} \ord(z) \le \gamma(\bar X + \bar Y).
    \end{displaymath}
    Furthermore, \eqref{equ:resume_su_X0_and_Y0} and \eqref{equ:last} give $|\bar X| = |X| + |X_0 + \bar z| = |X| + n_X > |X|$ and $|\bar Y| = |Y| - |Y_0| = |Y| - n_Y$, so $|\bar X| + |\bar Y| = |X| + |Y| + n_X - n_Y \ge |X| + |Y|$.
\vskip 0.1cm
\item[Case 2] $n_X < n_Y$. We set $\bar X := X \setminus X_0$ and $\bar Y := (\bar z + Y_0) \cup Y$. Then by repeating (except for obvious modifications) the same reasoning as in the previous case, we get again that $0 \in \bar X^\times \cap \bar Y^\times$ and $\bar X + \bar Y \subseteq X + Y$, with the result that $|\bar X + \bar Y| \le \kappa$ and $\gamma(X + Y) \le \gamma(\bar X + \bar Y)$. In addition, it follows from \eqref{equ:resume_su_X0_and_Y0} and \eqref{equ:last} that $|\bar X| = |X| - |X_0| = |X| - n_X$ and $|\bar Y| = |Y| + |\bar z + Y_0| = |Y| + n_Y$, whence $|\bar X| + |\bar Y| = |X| + |Y| + n_Y - n_X > |X| + |Y|$.
\end{description}
\vskip 0.1cm
So in both cases, we end up with an absurd, for we find subsets $\bar X$ and $\bar Y$ of $A$ that contradict the ``minimaximality'' of $(X,Y)$ as expressed by \eqref{equ:alternative_for_minimax}.
\end{proof}
It is perhaps worth noticing that several pieces of the above proof of Theorem \ref{th:new_strong_main} do \textit{not} critically depend on the cancellativity of the ambient, while others can be adapted to the case where $\gamma(X+Y)$ is replaced by $\gamma(X,Y)$. Nonetheless, this does not seem to be enough to prove Conjecture \ref{conj:cauchy_davenport} (not even for two summands).
\begin{proof}[Proof of Corollary \ref{cor:karolyi_generalized}]
The claim is obvious if $n = 1$. Thus, assume in what follows that $n$ is $\ge 2$ and the statement is true for all sumsets of the form $Y_1 + \cdots + Y_{n-1}$ with $Y_1^\times + \cdots + Y_{n-1}^\times \ne \emptyset$.
Then we get by Theorem \ref{th:new_strong_main} that
\begin{displaymath}
|X_1 + \cdots + X_n| \ge \min(\gamma(X_1 + \cdots + X_n), |X_1 + \cdots + X_{n-1}| + |X_n| - 1),
\end{displaymath}
which in turn implies, by Lemma \ref{lem:simple_lemma}, that
\begin{equation}
\label{equ:karolyi1}
|X_1 + \cdots + X_n| \ge \min(\mathfrak p(A), |X_1 + \cdots + X_{n-1}| + |X_n| - 1).
\end{equation}
But $X_1^\times + \cdots + X_{n-1}^\times \ne \emptyset$ by Proposition \ref{prop:units_distribution}, so our assumptions give
\begin{displaymath}
 |X_1 + \cdots + X_{n-1}| \ge \min(\mathfrak p(A), |X_1| + \cdots + |X_{n-1}| + 2 - n),
\end{displaymath}
which, together with \eqref{equ:karolyi1}, yields the desired conclusion (by induction).
\end{proof}
\section{Closing remarks}
While every \textit{commutative} cancellative semigroup embeds as
a subsemigroup into a group (through the standard construction of the group of fractions of a commutative monoid),
nothing similar is true in the non-commutative setting, no matter if the ambient semigroup is finitely generated or not. This is related to an old question in the theory of semigroups first answered by A.~I. Mal'cev in \cite{Malcev37}, and serves as a ``precondition'' for the present paper, as
it shows that the study of sumsets in cancellative semigroups cannot be systematically
reduced,
in the absence of commutativity, to the case of groups, at least not in an
obvious way (note that the semigroup considered by Mal'cev is also finitely generated).

On the other hand, every semigroup embeds into a monoid (e.g., through the unitization process of Section \ref{sec:preliminaries}), with the result that, \textit{for the specific purposes of the manuscript}, we could have assumed almost everywhere that the ``ambient'' is a monoid (rather than just a semigroup), but we did differently because, first, the assumption is not really necessary (insofar as Theorem \ref{th:new_strong_main} is vacuously true if $A^\times = \emptyset$), and second, it seems more appropriate to develop \textit{as much as possible} of the material with no regard to the presence of an identity (e.g., since this is better suited for further generalizations).

Next, we show with the following example that cancellativity is a somewhat necessary assumption for anything in the lines of Conjecture \ref{conj:cauchy_davenport} to be true:
\begin{example}
Let $X$ and $Y$ be non-empty disjoint sets with $|X| < \infty$, and denote by $(F_X, \cdot_X)$ and $(F_Y, \cdot_Y)$, respectively, the free abelian semigroups on $X$ and $Y$. For a fixed $e \notin F_X \cup F_Y$, we define a binary operation $\cdot$ on $F := F_X \cup F_Y \cup \{e\}$ by $u \cdot v := u \cdot_X v$ if $u,v \in F_X$, $u \cdot v := u \cdot_Y v$ if $u,v \in F_Y$, and $u \cdot v := e$ otherwise.

It is seen that $\cdot$ is associative; we write $\mathbb F$ for a unitization of $(F, \cdot)$ and $1$ for the identity of $\mathbb F$. Then taking $Z := Y \cup \{1\}$ gives $\gamma_\mathbb{F}(Z) = \infty$ and $X \cdot Z := \{x \cdot z: x \in X, z \in Z\} = X \cup \{e\}$, with the result that
$|X \cdot Z| < |X| + |Z| - 1 \le \gamma_\mathbb{F}(X,Z),
$
viz $|X \cdot Z| < \min(\gamma_\mathbb{F}(X,Z),|X| + |Z| - 1)$, where the right-hand side can be made arbitrarily larger than the left-hand side.
\end{example}

Lastly, we record here that if $\AA$ is cancellative and $X$ is a non-empty finite subsemigroup of $\AA$ with $|X| \ge 2$, then $|X+X| = |X| < 2|X|-1$, no matter whether or not $X$ contains units.

On the one hand, this means that the trivial lower bound provided by Proposition \ref{prop:trivialities}.\ref{item:trivialieties_(ii)} is sometimes better than the bound furnished by Theorem \ref{th:new_strong_main} (but this is, in some sense, a trivial observation). On the other hand, the problem arises of characterizing the pairs $(X,Y)$ of subsets of $A$ such that $|X+Y| < \gamma(X+Y)$, which we plan to investigate in future work.

\subsubsection*{Note added in proof:} It is possible to prove, as a corollary of Theorem \ref{th:old_strong_main}, that Conjecture \ref{conj:cauchy_davenport} is true when $\AA$ is a cancellative and commutative semigroup and $X_1, \ldots, X_n \subseteq A$ are such that $\gamma(X_1) = \cdots = \gamma(X_n)$, so that in particular $|nX| \ge \min(\gamma(X), n|X|+1-n)$ for every $X \subseteq A$ (actually, no restriction is necessary here on the size of the sets $X_i$).

\section{Acknowledgments}

The author is grateful to Carlo Sanna (Universit\`a di Torino, IT) for a detailed proof-reading of the paper, to Alain Plagne (CMLS, \'Ecole polytechnique, FR) for useful comments, and to an anonymous referee for many valuable suggestions.


\begin{thebibliography}{99}%
%
\bibitem{Cauchy1813} {A. L. Cauchy}, `Recherches sur les nombres', {\it J. \'Ecole
Polytech.} 9 (1813) 99--116.
%
\bibitem{Cill10} {J. Cilleruelo, Y. O. Hamidoune \and O. Serra}, `Addition theorems in acyclic semigroups.' In {\it Additive number theory} (Springer, 2010) 99--104.
%
\bibitem{Daven47} H. Davenport, `A historical note', {\it J. London Math. Soc.} 22 (1947) 100--101.
%
\bibitem{GerSch} A.~Geroldinger and W.~A.~Schmid, \textit{The system of sets of lengths in Krull monoids under set addition},
        preprint (\href{http://arxiv.org/abs/1407.1967}{arXiv:1407.1967}).
%
\bibitem{Hami09} {Y. O. Hamidoune}, `The isoperimetric method'. In A.~Geroldinger and I.~Z.~Ruzsa, {\it Combinatorial Number Theory and Additive Group Theory} (Birkh\"auser, 2009) 87--210.
%
\bibitem{Howie96} {J. M. Howie},
{\it Fundamentals of semigroup theory} (Clarendon Press, 1995).
%
\bibitem{Karo05.2} {G. K\'arolyi}, `The Cauchy-Davenport theorem in group
extensions', {\it Enseign. Math.} 51 (2005) 239--254.
%
\bibitem{Kemp56} {J. H. B. Kemperman}, `On complexes in a semigroup', {\it Indag.
Math.} 18 (1956) 247--254.
%
\bibitem{Malcev37} {A. I. Mal'cev}, `On the immersion of an algebraic ring
into a field', {\it Math. Annalen} 113 No. 1 (1937) 686--691.
%
\bibitem{Olson} {J. E. Olson}, `On the Sum of Two Sets in a Group', {\it J. Number Th.} 18 (1984), 110--120.
%
\bibitem{Ruzsa09} {I. Z. Ruzsa}, `Sumsets and structure.' In A.~Geroldinger and I.~Z.~Ruzsa, {\it Combinatorial Number Theory and Additive Group Theory} (Birkh\"auser, 2009) 87--210.
%
\bibitem{Tring12} {S. Tringali}, `A Cauchy-Davenport theorem for semigroups', {\it Unif. Distrib. Theory} 9 No. 1 (2014) 27--42.
%
\bibitem{Tring11} \bysame, `Small doubling in ordered semigroups', {\it Semigroup Forum} 90 No. 1 (2015) 135--148.
%
\end{thebibliography}
\end{document}